\documentclass[preprint,12pt,authoryear]{elsarticle}

\usepackage{amssymb}
\usepackage{graphicx}
\usepackage{enumerate}
\usepackage{amsthm}
\usepackage{amsfonts}
\usepackage{amsmath}
\usepackage{mathrsfs}
\usepackage{enumitem}
\usepackage{booktabs}
\usepackage{caption}
\usepackage{natbib}
\setcitestyle{numbers,square}
\usepackage{multirow}
\usepackage{tabularx}
\usepackage[all, cmtip]{xy}

\newcommand{\Tr}{{\rm Tr}}
\newcommand{\gf}{ {{\mathbb F}} }
\newtheorem{lemma}{Lemma}[section]
\newtheorem{definition}{Definition}[section]
\newtheorem{example}{Example}[section]
\newtheorem{theorem}{Theorem}[section]
\newtheorem{corollary}{Corollary}[section]

\newtheorem{question}{Question}[section]
\newtheorem{remark}[theorem]{Remark}

\begin{document}

\begin{frontmatter}



\title{Permutation polynomials over finite fields by the local criterion}


\author[wuaddress]{Danyao Wu\corref{mycorrespondingauthor}}
\cortext[mycorrespondingauthor]{Corresponding author}
\ead{wudanyao111@163.com}

\author[yuanaddress]{Pingzhi Yuan}
\ead{yuanpz@scnu.edu.cn}

\address[wuaddress]{School of Computer Science and Technology,
	Dongguan University of Technology, Dongguan 523808, China}
\address[yuanaddress]{School of Mathematics, South China Normal University, Guangzhou 510631, China}
%

%

\begin{abstract}
In this paper, we  further  investigate the local criterion and present a class of permutation polynomials and their compositional inverses  over $
\gf_{q^2}$. Additionally, we demonstrate that  linearized  polynomial over $\gf_{q^n}$ is a local permutation polynomial with respect to all linear transformations from $\gf_{q^n}$ to $\gf_q  ,$  and that  every permutation polynomial is a local permutation polynomial with respect to certain mappings.    
\end{abstract}


\end{frontmatter}

\section{Introduction}
\label{}
Let  $\gf_q$ be the finite field with $q$ elements and $\gf_{q}^*$ denote the
multiplicative group with the nonzero element in $\gf_q$, where $q$ is a prime power. Let $\gf_q[x]$
be the ring of polynomials in a single indeterminate $x$ over $\gf_q$. A polynomial
$f \in\gf_q[x]$ is called a {\em permutation polynomial}  of $\gf_q$ if its
associated polynomial mapping $f: c\mapsto f(c)$ from $\gf_q$ to itself is  bijective. The unique polynomial denoted by $f^{-1}(x)$ over $\gf_q$
such that $f(f^{-1}(x))\equiv f^{-1}(f(x)) \equiv x \pmod{x^q-x}$ is called the compositional inverse of $f(x).$ Furthermore,  $f(x)$ is called  an involution when $f^{-1}(x)=f(x).$

The study of permutation polynomials and their compositional inverses over finite
fields in terms of their coefficients is a classical and difficult subject which
attracts people's interest partially due to their wide applications in coding theory
\cite{ding2013cyclic,ding2014binary,laigle2007permutation},
cryptography \cite{rivest1978method,schwenk1998public}, combinatorial design theory \cite{ding2006family}, and other areas of mathematics and engineering \cite{lidl1997finite,lidl1994introduction}.

In general, determining whether a polynomial over a finite field is a permutation polynomial is challenging. Moreover, computing the coefficients of the compositional inverse of a permutation polynomial is even more difficult, except for several well-studied classes such as monomials, linearized polynomials, and Dickson polynomials, which prossess a well-defined structure.
We refer readers to \cite{li2019compositional,wu2024permutation,wu2024some,wu2022further,yuan2022agwpps,yuan2022local,yuan2022permutation,zheng2020on,zheng2020onieee} for more details.

In 2011, Akbrary, Ghioca and Wang \cite{AGW} introduced a powerful method known as  the AGW criterion for
constructing permutation polynomials. In 2024, the coauthor \cite{yuan2022local} proposed the local criterion, which is entirely  equivalent to the AGW criterion, but has fewer constraints. Utilizing the local criterion, the coauthor also developed a local method to find the compositional inverses of permutation polynomials. In this paper, we further investigate the properties of permutation polynomials using local criterion. 

The remainder of this paper is organized as follows. In Section 2, we refine the local criterion and  introduces some related results. In Section 3, we present a  class of permutation polynomials over $\gf_{q^2}$ by applying the refinement local criterion from Section 2 and further investigate the properties of linearized polynomials. Finally,  in Section 4, we demonstrate that  linearized  polynomial over $\gf_{q^n}$ is a local permutation polynomial with respect to all linear transformations from $\gf_{q^n}$ to $\gf_q$ and every permutation polynomial is a local permutation polynomial with respect to certain mappings.   
\section{Auxiliary results and the main Lemma}
In this section, we present some auxiliary results that will be needed in the sequel.
 
 First, we recall the local criterion.  

\begin{lemma}\label{local method}\cite[Lemma 2.1 ]{yuan2022local}(Local Criterion)
	Let  $A$ and $S$ be finite sets,  and let $f(x): A\rightarrow A$ be a map. Then $f(x)$ is a bijection if and only if for any surjection  $\psi(x): A \rightarrow S,$ the composition  $\varphi(x)= \psi(x) \circ f(x)$ is a surjection, and $f(x)$ is injective  on $\varphi^{-1}(s) $ for each $s \in S.$
	
\end{lemma}

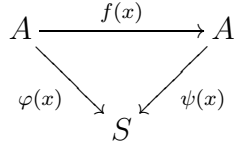
$$\xymatrix{
A \ar[rr]^{f(x)}\ar[dr]_{\varphi(x)}& & A\ar[dl]^{\psi(x)}\\
&S
}$$

%

\begin{corollary}\label{colene}
	Let $A $ and $S$ be finite sets, and let $\varphi(x): A \rightarrow S$ be a surjective map.
	If the map $f(x): A \rightarrow A $ is bijective, then $f(x)$ is injective on $\varphi^{-1}(x)$ for each $s \in S.$
\end{corollary}

According to  Corollary \ref{colene},  the requirement  that $f(x)$ is injective on $\varphi^{-1}(s)$ for each $s \in S$ is necessary when  discussing whether $f(x)$ is a bijection or not. Based on this,  we present the following refinement of the local criterion. 
\begin{theorem}\label{threlo}
	Let $A$ and $S$ be finite sets. Assume that $\varphi(x):A \rightarrow S$ and $f(x): A\rightarrow A $ are maps. 
	If $\varphi(x)$ is surjective and $f(x)$ is injective on $\varphi^{-1}(x)$ for each $s \in S,$ then $f(x)$ is bijective if and only if there exists a unique surjection $\psi(x): A \rightarrow S$ such that $\varphi(x)=\psi(x) \circ f(x).$ Moreover,  there are  $\prod_{s \in S}\left(|\varphi^{-1}(s)|!\right)$ bijections $f(x)$ determined by the map $\varphi(x).$
\end{theorem}
\begin{proof}
	The sufficiency is obvious by Lemma \ref{local method}; now, we prove the necessity.
	Let $S=\{s_1, s_2, \cdots, s_t\}.$ Since $\varphi(x)$ is a surjection from $A$ to $S,$ $A=\uplus_{i=1}^t\varphi^{-1}(s_i)$ is a disjoint union of the subsets $\varphi^{-1}(s_i), i=1, 2, \cdots, t.$ If $f(x)$ is bijective, then $f(\varphi^{-1}(s_i))\cap f(\varphi^{-1}(s_j))=\emptyset$ for distinct $s_i, s_j \in S,$ whence $| \varphi^{-1}(s_i) | = | f(\varphi^{-1}(s_i)) |, i=1, 2, \cdots, t.$ So $A=\uplus_{i=1}^tf(\varphi^{-1}(s_i))$ is also a disjoint union.
	
	 Now, we define a map 
	 $\psi(x): A\rightarrow S$ by $\psi(x)=s$ for any $s \in S$ and  $x \in f(\varphi^{-1}(s)).$   
	 It is easy to check that $\psi(x)$ is a surjection from $A$ to $S$ and $\varphi(x)= \psi(x) \circ f(x).$ Moreover, for any  $s \in S,$ $f(\varphi^{-1}(x))=\psi^{-1}(s)$, so $\psi(x)$ is uniquely determined by $f(x)$ and $\varphi(x).$
	 
	 On the other hand, for given surjections $\varphi(x): A\rightarrow S, \psi(x): A\rightarrow S,$ if $f(x)$ is bijective from $A$ to $A$ such that $\varphi(x)=\psi(x)\circ f(x),$ then $f(x)$ must be a bijection from $\varphi^{-1}(s)$ to $\psi^{-1}(s)$ for all $s \in S.$ Hence, the number of  $f(x)$ determined by $\varphi(x)$ and $\psi(x)$ is $\prod_{s \in S}\left(|\varphi^{-1}(s)|!\right).$
	 We are done. 	 
\end{proof}

%
Yuan \cite{yuan2022local} presented a local method to find the compositional inverses of permutation polynomials.

\begin{lemma}\label{ff-}\cite[Theorem 2.2]{yuan2022local}
	Let $q$ be a prime power and $f(x)$ be a polynomial over $\gf_q.$ Then
	$f(x)$ is a permutation polynomial over $\gf_q$ if and only if there exist nonempty finite
	subsets $S_i$, $i=1, 2, \cdots, t$ of $\gf_q$ and maps $\psi_i(x): \gf_q \rightarrow S_i$, $i=1, 2, \cdots, t$ such that $\psi_i(x)\circ f(x)=\varphi_i(x),$ $i=1, 2, \cdots, t$
	and $x=F(\varphi_1(x), \varphi_2(x), \cdots, \varphi_t(x)),$ where $F(x_1, x_2, \cdots, x_t)\in \gf_q[x_1, x_2, \cdots, x_t].$ Moreover, the compositional inverse of $f(x)$ is given by
	$$f^{-1}(x)=F(\psi_1(x), \psi_2(x), \cdots, \psi_t(x)).$$
\end{lemma}


The following lemma is crucial when applying the refined local criterion in construct a class of permutation polynomials in Section 2.

\begin{lemma}\cite[Lemma 2.3]{yuan2015permutation}\label{lex^q}
For a prime power $q,$ assume that $a \in \gf_{q^2}$  with $a^{q+1}=1$ and $g\in \gf_{q^2}$ is a primitive element of $\gf_{q^2}.$ Let $t$ be the non-negative  integer $a=g^{q+1}$ and $0\leq t\leq q.$ Let $\varphi(x)=x^q+ax$ and $\psi(x)=ax^q+x$ be polynomials over $\gf_{q^2}.$
Then ${\rm Im}(\varphi(x))={\rm Im}(\psi(x))=\{g^{-t}b \mid b \in \gf_q\}.$
  \end{lemma}

Based on the lemma above, we determine  whether the  equation $x^q+ax=d,$ where $a^{q+1}=1$,  has roots over  $\gf_{q^2}$ or not. This analysis is crucial for assessing the  number of permutation polynomials in Section 2. 

\begin{lemma}\label{lenu}
	For a prime power $q,$ assume that $a\ in \gf_{q^2}$  with $a^{q+1}=1.$
Then the affine $q$-polynomial  $x^q+ax-d$ has $q$ roots if and only if $ad^q=d.$ 
\end{lemma}
\begin{proof}
Since $a^{q+1}=1,$ we get  $(x^q+ax)^q=a^{-1}(x^q+ax)$.  It follows  easily  that ${\rm Im}(x^q+ax)=\{y \mid ay^q=y\}.$ Then we conclude the desired result. 	
\end{proof}
The trace function $\Tr(x)$ from $\gf_{q^n}$ to $\gf_q$ is defined by 
 $$\Tr(x)=x+x^q+\cdots+x^{q^{n-1}}.$$

Every linearized polynomial over $\gf_{q^n}$ can be written as $$L(x)=\Tr(\theta_1x)\omega_1+\Tr(\theta_2x)\omega_2 +\cdots+\Tr(\theta_nx)\omega_n \,\, \text{with}\,\, \omega_i\in\gf_{q^n},$$ where $\{\theta_1, \theta_2,\cdots, \theta_n\}$ forms a basis of $\gf_{q^n}$ over $\gf_{q}$.
The following  lemma  is important in further investigating the properties of linearized permutation polynomials. 


\begin{lemma}\cite[Theorem 1.2]{yuan2011note}\label{lemmalinearized}
	Let $\{\theta_1, \theta_2,\cdots, \theta_n\}$ be any given basis of $\gf_{q^n}$ over $\gf_{q}$, and let $$L(x)=\Tr(\theta_1x)\omega_1+\Tr(\theta_2x)\omega_2 +\cdots+\Tr(\theta_nx)\omega_n \,\, \text{with}\,\, \omega_i\in\gf_{q^n}.$$
	Then $L(x)$ is a permutation polynomial over $\gf_{q^n}$ if and only if  $\{\omega_1, \omega_2, \cdots, \omega_n \}$ is a  basis of $\gf_{q^n}$ over $\gf_{q}.$
\end{lemma}
The following lemma explains how to determine whether a set of elements forms a basis of $\gf_{q^n}$ over $\gf_q$ or not. 
\begin{lemma}\cite[Lemma 3.51]{lidl1997finite}\label{lemma 3.51}
	Let $\beta_1, \beta_2, \cdots, \beta_n $
	be elements of $\gf_{q^m}.$ Then 

	\end{lemma}
$\left|
\begin{array}{ccccc}
\beta_1 & \beta_1^q & \beta_1^{q^2}&\cdots&\beta_1^{q^{n-1}} \\
\beta_2 & \beta_2^q & \beta_2^{q^2}&\cdots&\beta_2^{q^{n-1}} \\
\vdots&\vdots&\vdots&&\vdots\\
\beta_n & \beta_n^q & \beta_n^{q^2}&\cdots&\beta_n^{q^{n-1}} \\
\end{array}
\right|=\beta_1\prod_{j=1}^{n-1}\prod_{c_1, \cdots, c_j \in \gf_q}\left(\beta_{j+1}-\sum_{k=1}^jc_k\beta_k\right),$
and so the determinant is $\neq 0$ if and only if $\beta_1, \beta_2, \cdots, \beta_n $ are linearly independent over $\gf_q.$

\begin{lemma}\cite[Theorem 2.24]{lidl1997finite}\label{lemma2.24}
	Let $\gf_{q^n}$ be a finite extension of the  field $\gf_q,$ both considered as vector spaces over $\gf_q.$ Then the linear transformations  from $\gf_{q^n}$ into $\gf_q$ are exactly the mappings $L_{\beta}, \beta \in \gf_{q^n},$ where $L_{\beta}(x)=\Tr(\beta x)$ for all $x \in \gf_{q^n}.$ Furthermore, we have $L_{\beta} \neq L_{\gamma} $ whenever $\beta$ and $\gamma $ are distinct elements of $\gf_{q^n}.$
\end{lemma}

\section{Permutation polynomials over finite fields  }
In this section, we construct a class of permutation polynomial of the form $f(x)=ux^q+vx+g(x^q+ax)$ over $\gf_{q^2}$ and further investigate the properties of linearized polynomials.
\begin{theorem}\label{thq^2}
	Let $q$ be a prime power. Assume that $a, u, v  \in \gf_{q^2}$ with $a^{q+1}=1, au-v\neq0$ and $g(x)=b_1x+b_2x^2+\cdots+b_{q-1}x^{q-1}.
	$ If one of the following conditions holds:\\
	(i) $\left(au+v^q+b_1^qa^{-1}+b_1a, u^q+av+a(b_1^qa^{-1}+b_1a)\right)=(c, ac)$ with some $c \in \gf_q^*,$ and  $b_i^qa^{-i}+b_ia=0,$ $i =2, 3 \cdots, q-1;$ \\
	(ii) $\left(u+av^q+b_1^q+b_1, au^q+v+a(b_1^q+b_1)\right)=(c, ac)$ with some $c \in \gf_q^*,$ and  $b_i^qa^{-i+1}+b_i=0,$ $i =2, 3 \cdots, q-1;$ \\	
then the polynomial $$f(x)=ux^q+vx+g(x^q+ax)$$ is a permutation polynomial  over $\gf_{q^2}.$ 

Moreover, if $\left(au+v^q+b_1^qa^{-1}+b_1a, u^q+av+a(b_1^qa^{-1}+b_1a)\right)=(c, ac)$ with some $c \in \gf_q^*,$ and  $b_i^qa^{-i}+b_ia=0,$ $i =2, 3 \cdots, q-1,$  then the compositional inverse of $f(x)$ over $\gf_{q^2}$ is given by $$f^{-1}(x)=(v-au)^{-1}\left(x-g\left(c^{-1}x^q+c^{-1}ax\right)-uc^{-1}(x^q+ax)\right).$$

If $\left(u+av^q+b_1^q+b_1, au^q+v+a(b_1^q+b_1)\right)=(c, ac)$ with some $c \in \gf_q^*,$ and  $b_i^qa^{-i+1}+b_i=0,$ $i =2, 3 \cdots, q-1,$  then the compositional inverse of $f(x)$ over $\gf_{q^2}$ is given by $$f^{-1}(x)=(v-au)^{-1}\left(x-g\left(c^{-1}ax^q+c^{-1}x\right)-uc^{-1}(ax^q+x)\right).$$ 
	\end{theorem}
\begin{proof}
First, we consider the second case. 
		
	Let $\varphi_1(x)=cx^q+acx$ and $\psi_1(x)=ax^q+x.$ For  $i =2, \cdots, q-1,$ since $b_i^qa^{-i+1}+b_i=0$  and $(x^q+ax)^q=a^{-1}(x^q+ax),$ and since moreover $\left(u+av^q+b_1^q+b_1, au^q+v+a(b_1^q+b_1)\right)=(c, ac)$ with $c \in \gf_q^*$, we have 
	\begin{align}\label{eqq2}		&\psi_1(x)\circ f(x)\nonumber\\
		=&\, (ax^q+x) \circ \left(ux^q+vx+b_1(x^q+ax)+b_2(x^q+ax)^2+\cdots+b_{q-1}(x^q+ax)^{q-1}\right)\nonumber\\
		=&\, (u+av^q)x^q+(au^q+v)x+(b_1^q+b_1)(x^q+ax)+(b_2^qa^{-1}+b_2)(x^q+ax)^2 \nonumber\\
		&+(b_3^qa^{-2}+b_3)(x^q+ax)^3+\cdots+(b_{q-1}^qa^{-(q-2)}+b_{q-1})(x^q+ax)^{q-1}
		\nonumber \\
		=&\, (u+av^q+b_1^q+b_1)x^q+(au^q+v+a(b_1^q+b_1))x\nonumber\\
		=&\, c(x^q+ax)\nonumber\\
		=&\, \varphi_1(x).
	\end{align}
It follows from Lemma \ref{lex^q} that  ${\rm Im}(x^q+ax)={\rm Im}(\varphi_1(x))={\rm Im}(\psi_1(x))=\{g^{-t}b \mid b \in \gf_q\}.$ Consequently, $\psi_1(x)$ is a surjection from $\gf_{q^2}$ to ${\rm Im}(\varphi_1(x)).$ 

Moreover, for any $s \in {\rm Im}(\varphi_1(x)),$ if there exist $x_1, x_2 \in \varphi_1^{-1}(s)$ such that $f(x_1)=f(x_2),$
then  it follows that 
\begin{equation*}\begin{cases}
	x_1^q+ax_1&=x_2^q+ax_2,\\
	ux_1^q+vx_1&=ux_2^q+vx_2.
\end{cases}
\end{equation*}
This implies $x_1=x_2$ because $au-v\neq0.$ Therefore, $f(x)$ is injective on $\varphi_1^{-1}(s)$ for each $s \in {\rm Im}(\varphi_1^{-1}(s)).$

	It follows from Theorem \ref{threlo} that $f(x)$ permutes $\gf_{q^2}.$ Furthermore,  let $\varphi_2(x)=x\circ f(x)$ and  $\psi_2(x)=x.$ Note that  $$(v-au)^{-1}\left(\varphi_2(x)-g\left(c^{-1}\varphi_1(x)\right)-uc^{-1}\varphi_1(x)\right)=x.$$
	It implies by Lemma \ref{ff-} that the compositional inverse of $$f^{-1}(x)=(v-au)^{-1}\left(x-g\left(c^{-1}ax^q+c^{-1}x\right)-uc^{-1}(ax^q+x)\right).$$ 
	In the first case, we will  prove $ (x^q+ax) \circ f(x) =c(x^q+ax).$ The remaining proof follows similarly, we omits the details. 
	This completes the proof.  
\end{proof}
\begin{remark}
Let the notations be defined in Theorem \ref{thq^2}. Now, we consider the count of $f(x)$  in Theorem \ref{thq^2}. We examine the solutions of the following system of equations
\begin{equation*}
	\begin{cases}
		au-v&\neq0,\\
		u+av^q+b_1^q+b_1&=c,\\
		au^q+v+a(b_1^q+b_1)&=ac,\\
		b_i^qa^{-i+1}+b_i&=0,		
	\end{cases}
\end{equation*}
which is equivalent to 
\begin{equation*}
	\begin{cases}
		au-v&\neq0,\\
		a(u+av^q)&=au^q+v,\\
		u+av^q+b_1^q+b_1&=c,\\
		b_i^qa^{-i+1}+b_i&=0.		
	\end{cases}
\end{equation*}
We rewrite the above system as 
\begin{equation*}
		\begin{cases}
		au-v&\neq0,\\
		u^q-u&=av^q-a^{-1}v,\\
		b_1^q+b_1&=c-u-av^q,\\
		b_i^qa^{-i+1}+b_i&=0.		
	\end{cases}
	\end{equation*}

Given $a^{q+1}=1$ and 	$a(u+av^q)=au^q+v,$ we derive  the identities  $-(av^q-a^{-1}v)^q=av^q-a^{-1}v$ and $(c-u-av^q)^q=c-u-av^q.$  According to Lemma \ref{lenu},  for any $v\in \gf_{q^2}$ with $au-v\neq0$, there are  $q-1$ possible choices for $u.$ 
Furthermore,  for any $c \in \gf_q^*,$ $u, v \in \gf_{q^2}$ with  $au-v\neq0$ and $
u^q-u=av^q-a^{-1}v,$ the number of  choices for  $b_1$ is $q$, and for $i=2, 3 , \cdots q-1$, the number of choices of $b_i$ are also $q.$

Consequently, the number of permutation polynomials $f(x)$ satisfying the second condition is $q^{q+2}(q-1)^2.$ Similarly, in the first case, the number of permutation polynomials $f(x)$ is also $q^{q+2}(q-1)^2.$

However,  Theorem \ref{threlo} implies that there are $(q!)^q$ permutation polynomials over $\gf_q$ determined by the map $\varphi(x)=x^q+ax,$ which is significantly greater than $2q^{q+2}(q-1)^2.$  Therefore,  many permutation polynomials determined by $x^q+ax$ remain undiscovered. 
	\end{remark}

Next, we further study the properties of linearized polynomials.  
Considered as maps between finite fields, linearized polynomials are always taken as 
\begin{equation}\label{linearized}
	L(x)=\sum_{i=0}^n a_ix^{q^i} \in \gf_{q^n}[x]/(x^{q^n}-x)
\end{equation}
We denote by $\mathscr{L}(\gf_{q^n})$ the set of all linearized polynomials in the form \eqref{linearized}.

A well known result of Dickson indicates that $L(x)$ in the form \eqref{linearized} is a linearized permutation polynomial if and only if the associated Dickson matrix 
\begin{equation}
	D_L=\left(
	\begin{array}{cccc}
	a_0&a_1&\cdots&a_{n-1}\\
	a_{n-1}^q&a_0^q&\cdots&a_{n-2}^q\\
	\vdots&\vdots&&\vdots\\
	a_1^{q^{n-1}}&a_2^{q^{n-1}}&\cdots&a_0^{q^{n-1}}
	\end{array}\right)
	\end{equation}
is non-singular.

Wu \cite{wu2013linearized} characterized $\mathscr{L}(\gf_{q^n})$ by the Dickson matrix algebra approach and derived some relations between linearized polynomials and their associated Dickson matrices. 
\begin{theorem}\cite[Theorem 4.5]{wu2013linearized}
	Let $L(x)=\sum_{i=0}^{n-1}a_ix^{q^i} \in \mathscr{L}(\gf_{q^n}) $ be a linearized permutation polynomial and $D_L$ be its associated Dickson matrix. Assume $\bar{a}_i$ is the $(i, 0)-$the cofactor of $D_L$, $0 \leq i\leq n-1.$ Then ${\rm det}\,L=\sum_{i=0}^{n-1}a_{n-i}^{q^i}\bar{a}_i$ 
	and $$L^{-1}(x)=\frac{1}{{\rm det}\,L}\sum_{i=0}^{n-1}\bar{a}_ix^{q^i}=\left(\sum_{i=0}^{n-1}\bar{a}_ix^{q^i}\right) \circ \left(\frac{x}{{\rm det}\,L}\right).$$
\end{theorem}

\begin{remark}
Indeed, if $L(x)$ is a permutation polynomial over $\gf_{q^n}, $ we can derive the compositional inverse of $L(x)$ by the local method. 
Since the associated Dickson matrix $D_L$ is non-singular and 
\begin{equation*}
	\left(
	\begin{array}{c}
		L(x)\\
			L(x)^q\\
			\vdots\\
				L(x)^{q^n-1}
	\end{array}
	\right)=D_L
	\left(
	\begin{array}{c}
		x\\
		x^q\\
		\vdots\\
	x^{q^n-1}
	\end{array}
	\right), 
\end{equation*}
 we have 
\begin{equation*}
D_L^{-1}	\left(
	\begin{array}{c}
		L(x)\\
		L(x)^q\\
		\vdots\\
		L(x)^{q^n-1}
	\end{array}
	\right)=
	\left(
	\begin{array}{c}
		x\\
		x^q\\
		\vdots\\
		x^{q^n-1}
	\end{array}
	\right), 
\end{equation*}
or \begin{equation*}
\frac{	D_L^{*}}{{\rm det}\, L}	\left(
	\begin{array}{c}
		L(x)\\
		L(x)^q\\
		\vdots\\
		L(x)^{q^n-1}
	\end{array}
	\right)=
	\left(
	\begin{array}{c}
		x\\
		x^q\\
		\vdots\\
		x^{q^n-1}
	\end{array}
	\right), 
\end{equation*}
where $D_L^{*}$ is the adjugate matrix of $D_L.$\\
Therefore, $$\frac{1}{{\rm det}\,L}\sum_{i=0}^{n-1}\bar{a}_iL(x)^{q^i}=x.$$
Hence,  it follows from Lemma \ref{ff-} that the compositional inverse of $L(x)$ is given by 
$$L^{-1}(x)=\frac{1}{{\rm det}\,L}\sum_{i=0}^{n-1}\bar{a}_ix^{q^i}.$$
\end{remark}
We  derive another  result concerning linearized polynomials using a local method. 
\begin{theorem}
	Let $\{\theta_1, \theta_2,\cdots, \theta_n\}$ be any given basis of $\gf_{q^n}$ over $\gf_{q}$, and let $$L(x)=\Tr(\theta_1x)\omega_1+\Tr(\theta_2x)\omega_2 +\cdots+\Tr(\theta_nx)\omega_n, \,\, \text{with}\,\, \omega_i\in\gf_{q^n}.$$
	Then $L(x)$ is a permutation polynomial over $\gf_{q^n}$ if and only if 
	$$D_1=\left|
	\begin{array}{cccc}
		\Tr(\theta_1\omega_1)&	\Tr(\theta_1\omega_2)&\cdots&	\Tr(\theta_1\omega_n)\\
		\Tr(\theta_2\omega_1)&	\Tr(\theta_2\omega_2)&\cdots&	\Tr(\theta_2\omega_n)\\
		\vdots&\vdots&&\vdots\\
		\Tr(\theta_n\omega_1)&	\Tr(\theta_n\omega_2)&\cdots&	\Tr(\theta_n\omega_n)\\
	\end{array}	
	\right|\neq0.$$
\end{theorem}
\begin{proof}
	Since $\Tr(x)$ is a function from $\gf_{q^n}$ to $\gf_q,$ we have 
	\begin{align*}
		\Tr(\theta_i x)\circ L(x)=&\, \Tr\left(\theta_i\Tr(\theta_1x)\omega_1+\theta_i\Tr(\theta_2x)\omega_2 +\cdots+\theta_i\Tr(\theta_nx)\omega_n\right)\\
		=&\, \sum_{j=1}^{n}\Tr(\theta_jx)\Tr(\theta_i\omega_j)\\
		=&\, \Tr\left(\sum_{j=1}^n\theta_j\Tr(\theta_i\omega_j)x\right).
	\end{align*}
	For simplicity, put $\eta_i=\sum_{j=1}^n\theta_j\Tr(\theta_i\omega_j).$  
	Then we have 
	\begin{equation*}
		\left(
		\begin{array}{c}
			\Tr(\theta_1L(x))\\
			\Tr(\theta_2L(x))\\
			\vdots\\
			\Tr(\theta_nL(x))\\
		\end{array}
		\right)=
		\left(	\begin{array}{c}
			\Tr(\eta_1x)\\
			\Tr(\eta_2x)\\
			\vdots\\
			\Tr(\eta_nx)\\
		\end{array}
		\right)=D_2\left(
		\begin{array}{c}
			x\\
			x^q\\
			\vdots\\
			x^{q^{n-1}}
		\end{array}
		\right),
	\end{equation*}
	where 
	\begin{equation*}
		D_2=	\left(
		\begin{array}{ccccc}
			\eta_1 & \eta_1^q & \eta_1^{q^2}&\cdots&\eta_1^{q^{n-1}} \\
			\eta_2 & \eta_2^q & \eta_2^{q^2}&\cdots&\eta_2^{q^{n-1}} \\
			\vdots&\vdots&\vdots&&\vdots\\
			\eta_n & \eta_n^q & \eta_n^{q^2}&\cdots&\eta_n^{q^{n-1}} \\
		\end{array}
		\right).
	\end{equation*}
	Therefore,  $x$ can be uniquely represented by $	\Tr(\theta_1L(x)),$ $ 
	\Tr(\theta_2L(x)), $ $
	\cdots, $ $
	\Tr(\theta_nL(x))$ if and only if the matrix $D_2$ is non-singular. According to  Lemma \ref{lemma 3.51}, the non-singularity of   $D_2$  is equivalent to $\{\eta_1, \eta_2,\cdots, \eta_n\}$  forming a basis of $\gf_{q^n}$ over $\gf_q.$  Since $\{\theta_1, \theta_2,\cdots, \theta_n\}$ is a   basis of $\gf_{q^n}$ over $\gf_{q},$ 
	$\{\eta_1, \eta_2,\cdots, \eta_n\}$  forms a basis of $\gf_{q^n}$ over $\gf_{q}$ if only and if  $D_1\neq0.$
	
	From Lemma \ref{ff-},   it follows  that $L(x)$ permutes $\gf_{q^n}$ if and only if $D_1\neq0.$ This establishes the desired result. 
\end{proof}

\section{Local permutation polynomials with respect to special  maps over finite fields }
In this section, we investigate the properties of permutation polynomial over finite fields. 

Now, we provide an example of permutation polynomials over finite fields. 
\begin{example} \cite{park2001permutations,wang2007cyclotomic,zieve2009}
	Pick $h(x) \in \gf_q[x]$  and integers $r, s >0$ such that $s \mid (q-1).$ Then $f(x)=x^rh(x^s)$ permutes $\gf_q$ if and only if \\	
	(i) $\gcd(r, s)=1$ and; \\	
	(ii) $x^rh(x)^s$ permutes the set of $(q-1)/s$-th roots of unity in $\gf_q^*.$
	
		Let $\varphi(x)=x^s.$ For  any $u \in {\rm Im}(\varphi(x))$, the condition that $f(x)=x^rh(x^s)$ is injective on $\varphi^{-1}(u)$ is equivalent to $\gcd(r, s)=1.$   By Corollary \ref{colene}, 
	the condition  $\gcd(r, s)=1$ is necessary when determining whether $f(x)=x^rh(x^s)$ is a permutation polynomial over $\gf_q$  or not.  
	
	We see  that the polynomial $f(x)=x^rh(x^s)$ over $\gf_q$ with $\gcd(r, s)=1$ is a permutation polynomial 
	over $\gf_q$  if and only if  $x^s \circ \left(x^rh(x^s)\right)$ 
	is a surjection from $\gf_q$ to ${\rm Im}(x^s)$.
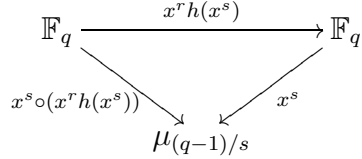
		$$\xymatrix{
		\gf_q \ar[rr]^{x^rh(x^s)}\ar[dr]_{x^s \circ  \left(x^rh(x^s)\right)}& & \gf_q\ar[dl]^{x^s}\\
		&\mu_{(q-1)/s}
	}$$

The necessity of the condition is already noted. 	
	Conversely,   if $x^s \circ \left(x^rh(x^s)\right)$ 
is a surjection from $\gf_q$ to ${\rm Im}(x^s)$, let $\mu_{(q-1)/s}$  denote the set of $(q-1)/s$-th roots of unity in $\gf_q^*.$ 
Suppose, on the contrary,  if $f(x)$ is not a permutation polynomial over $\gf_q$, then $x^rh(x)^s$ does not permute $\mu_{(q-1)/s},$ and so there exists $y \in  \mu_{(q-1)/s}$ such that for any $x_1 \in \mu_{(q-1)/s},$ $y \neq x_1^rh(x_1)^s.$ This implies that for any $x\in \gf_q$ such that  $x^s\in \mu_{(q-1)/s},$ we have $y \neq (x^s)^rh(x^s)^s, $ Hence,  $x^s \circ \left(x^rh(x^s)\right)$ 
is not a surjection from $\gf_q$ to ${\rm Im}(x^s),$ contrary to assumption.

	\end{example}

Next, we aim to demonstrate  that every linearized  polynomial $L(x)$ possesses a  similar property. Specifically, for certain mappings  $\psi_i(x)$ ($i=1, 2, \cdots, t$) where   the compositions $\psi_i(x)\circ f(x)$ are  surjective from $\gf_{q^n}$ to ${\rm Im}(\psi_i(x))$, it follows that  $L(x)$ is a permutation polynomial.  We begin with the following lemma. 
\begin{lemma}\label{lebasis}
Let $\omega_1, \omega_2, \cdots, \omega_n$ be $n$ elements of $\gf_{q^n}$ over $\gf_q. $ 
Then for any $u \in \gf_{q^n}^*$, $\Tr(u\omega_i)\neq 0$ if and only if $\{\omega_1, \omega_2, \cdots, \omega_n\}$ is a basis of $\gf_{q^n}$ over $\gf_q. $
\end{lemma}

\begin{proof}
	For any $u \in \gf_{q^n}^*$, it holds that $\Tr(u\omega_i)\neq 0.$ 
	We proceed to demonstrate  that $\{\omega_1, \omega_2, \cdots, \omega_n\}$ forms a basis of $\gf_{q^n}$ over $\gf_q. $ 
	Indeed, suppose otherwise. Then the $n$ elements $\omega_1, \omega_2, \cdots, \omega_n$ are linearly dependence over $\gf_q.$ We can assume, without lose of generality, that $\{\omega_1, \omega_2, \cdots, \omega_t\}$ is maximal independent subset of $\{\omega_1, \omega_2, \cdots, \omega_n\}$  and 
	\begin{equation}\label{lemmabasis}
	\omega_{t+1}=c_1\omega_1+c_2\omega_2+\cdots+c_t\omega_t, \,\, \text{with  $c_i \in \gf_q$ not all being $0$ }.
	\end{equation}
	If $\{\omega_1,  \cdots, \omega_t, \omega_{t+1}', \cdots, \omega_{n}'\}$
	is a basis of $\gf_{q^n}$ over $\gf_q$ extended by 
	$\{\omega_1, \cdots, \omega_t\},$
	then there exists its corresponding dual basis $\{\alpha_1, \alpha_2,\cdots,  \alpha_n\}.$
	So, by \eqref{lemmabasis},
	\begin{align*}
	\Tr(\alpha_n \omega_{t+1})=&\,\Tr\left(\alpha_n (c_1\omega_1+c_2\omega_2+\cdots+c_t\omega_t)\right)\\
	=&\,\sum_{i=1}^tc_i\Tr(\alpha_n\omega_i)\\
	=&\,0,
		\end{align*}
	which is a contradiction.

Conversely, if $\{\omega_1, \omega_2, \cdots, \omega_n\}$ is a basis of $\gf_{q^n}$ over $\gf_q, $ then by Lemma \ref{lemma 3.51}, the matrix 
	\begin{equation*}
	\left(
		\begin{array}{ccccc}
				\omega_1 & \omega_1^q & \omega_1^{q^2}&\cdots&\omega_1^{q^{n-1}} \\
				\omega_2 & \omega_2^q & \omega_2^{q^2}&\cdots&\omega_2^{q^{n-1}} \\
				\vdots&\vdots&\vdots&&\vdots\\
				\omega_n & \omega_n^q & \omega_n^{q^2}&\cdots&\omega_n^{q^{n-1}} \\
			\end{array}
		\right)
		\end{equation*} is non-singular. This implies that  the system of linear equations 
\begin{equation*}
	\left(
	\begin{array}{c}
		\Tr(\omega_1x)\\
			\Tr(\omega_2x)\\
		\vdots\\
			\Tr(\omega_nx)\\
		\end{array}	
	\right)=	
	\left(
	\begin{array}{ccccc}
			\omega_1 & \omega_1^q & \omega_1^{q^2}&\cdots&\omega_1^{q^{n-1}} \\
			\omega_2 & \omega_2^q & \omega_2^{q^2}&\cdots&\omega_2^{q^{n-1}} \\
			\vdots&\vdots&\vdots&&\vdots\\
			\omega_n & \omega_n^q & \omega_n^{q^2}&\cdots&\omega_n^{q^{n-1}} \\
		\end{array}
	\right)\left(
	\begin{array}{c}
			x\\
			x^q\\
			\vdots\\
			x^{q^{n-1}}
		\end{array}
	\right)=\left(
	\begin{array}{c}
			0\\
			0\\
			\vdots\\
			0
		\end{array}
	\right)
\end{equation*}
has no non-zero solution over $\gf_{q^n}^{*n}$. This indicates  that for any $u \in \gf_{q^n}^*,$ $\Tr(u\omega_i)\neq0.$
 We are done.  
	\end{proof}

\begin{theorem}\label{linearizedlocal}
	Let $\{\theta_1, \theta_2,\cdots, \theta_n\}$ be any given basis of $\gf_{q^n}$ over $\gf_{q}$, and let $$L(x)=\Tr(\theta_1x)\omega_1+\Tr(\theta_2x)\omega_2 +\cdots+\Tr(\theta_nx)\omega_n, \,\, \text{with}\,\, \omega_i\in\gf_{q^n}.$$
Then $L(x) $ is a permutation polynomial over $\gf_{q^n}$ if and only if  for any $u\in \gf_{q^n}^*$, $\Tr(u x) \circ L(x) \not\equiv 0\pmod{x^{q^n}-x}.$
\end{theorem}

\begin{proof}For any $u\in \gf_{q^n}^*$, we have 	
	\begin{align*}
		\Tr(ux)\circ L(x)=\Tr(u L(x))&=\Tr\left(u\sum_{i=1}^n\omega_i\Tr(\theta_ix)\right)\\
		&=\sum_{i=1}^n\Tr(\theta_ix)\Tr(\omega_iu)\\
		&=\sum_{i=1}^n\Tr\left(\theta_i\Tr(\omega_iu)x\right)\\
		&=\Tr\left(\left(\sum_{i=1}^n\Tr(\omega_iu)\theta_i\right)x\right).
	\end{align*}
	This yields that  $\Tr(ux)\circ L(x)\not\equiv 0\pmod{x^{q^n}-x}$ if and only if $\sum_{i=1}^n\Tr(\omega_iu)\theta_i\ne0$. Since $\{\theta_1, \theta_2,\cdots, \theta_n\}$ is a basis of $\gf_{q^n}$ over $\gf_{q}$, $\sum_{i=1}^n\Tr(\omega_iu)\theta_i\ne0$ if and only if $\Tr(\omega_iu)\ne0$ for some $i\in\{1, \dots, n\}$, i.e., there does not exist  $u\in \gf_{q^n}^*$ such that 
	$$\Tr(\omega_iu)=0, \quad i=1, 2, \dots, n,$$
	which implies that $\sum_{i=1}^n\Tr(\omega_iu)\theta_i\ne0$ if and only if $\{\omega_1, \ldots, \omega_n\}$ is a basis of $\gf_{q^n}$ over $\gf_{q}$ by Lemma \ref{lebasis}, and the latter condition is equivalent to that  $L(x)$ is a permutation polynomial over $\gf_{q^n}$ by Lemma 2.5. 
	This completes the proof. 
\end{proof}
In Theorem \ref{linearizedlocal}, we employ all linear transformations from $\gf_{q^n}$ to $\gf_q,$ as described in Lemma \ref{lemma2.24},  to establish that $L(x)$ is a permutation polynomial over $\gf_{q^n}.$
Now, let us consider  an example where, despite finding maps $\psi_i(x)$ such that the compositions $\psi_i(x)\circ f(x)$ are surjections from $\gf_{q^n}$ to ${\rm Im}(\psi_i(x))$, the polynomial $f(x)$ dose  not permute $\gf_{q^n}.$
\begin{example}\label{exline}
	Let $\{\theta_1, \theta_2,\cdots, \theta_n\}$ and $\{v_1, v_2, \cdots , v_n\}$ be any given basis of $\gf_{q^n}$ over $\gf_{q}$. Assume that $$L(x)=\omega\sum_{i=1}^n\Tr(\theta_ix),$$
	where $\omega=a_1v_1+a_2v_2+\cdots+a_nv_n$ with $a_j\in \gf_q^*,$ $j=1, 2, \cdots, n.$
	Let $\{e_1, e_2, \cdots e_n\}$ be the dual basis of $\{v_1, v_2, \cdots , v_n\}.$ 
	We have 
	\begin{align}\label{eqline1}
		\Tr(e_jx) \circ L(x)=&\, \Tr\left(e_j\omega\sum_{i=1}^n\Tr(\theta_ix)\right)\nonumber\\
		=&\, \Tr(e_j\omega)\sum_{i=1}^n\Tr(\theta_ix)\nonumber\\
		=&\, \Tr\left(\sum_{i=1}^n\theta_i\Tr(e_j\omega)x\right)\nonumber\\
		=&\, \Tr\left(\sum_{i=1}^n\theta_ia_jx\right).
	\end{align}
	Since $\{\theta_1, \theta_2,\cdots, \theta_n\}$ is a basis of $\gf_{q^n}$ over $\gf_{q}$ and $a_j\in \gf_q^*,$ for $j=1, 2, \cdots, n,$ we have $\sum_{i=1}^n\theta_ia_j\neq0$ for all $j.$ Consequently, $\Tr(e_jx) \circ L(x)$
	are surjective from $\gf_{q^n}$  to $\gf_q$ as indicated by Eq. \eqref{eqline1}.
	Moreover,  every element in $\gf_{q^n}$ can be uniquely represented by $\Tr(e_1x),$ $ 
	\Tr(e_2x), $ $
	\cdots, $ $
	\Tr(e_nx).$ 
	
	However, despite these properties,  $L(x)$ is not a permutation polynomial over $\gf_{q^n},$ as shown in  Lemma \ref{lemmalinearized},  and the image of $L(x)$ forms a vector space over $\gf_q$ with dimension 1.
\end{example}
Example \ref{exline} demonstrates that    the  surjections  $\psi_{i}(x)$ such that the compositions $\psi_i(x) \circ f(x)$ are surjective from $\gf_{q^n}$ to ${\rm Im}(\psi_i(x))$ must not be too limited when establishing that  the linearized polynomial is a permutation polynomial  over finite fields.  This leads to the following question. 

\begin{question}
	What is the minimum number of polynomials of the form $\Tr(ax)$ required to conclusively prove that a linearized polynomial is a permutation polynomial?
\end{question}

Next, we will show that every permutation polynomial  has such property.
\begin{theorem}
	Let $A$ and $S_i$ $(i=1, 2, \cdots, t)$ be finite sets,  and let $\psi_i(x): A\rightarrow S_i$ be surjective maps and $f(x): A \rightarrow A $  be a map.
	Suppose that  $f(x)$ is a permutation polynomial if and only if 
	$\varphi_i(x)=\psi_i(x) \circ f(x)$ $(i=1, 2, \cdots t)$ are surjective from $A$ to $S_i$. Then for any permutation polynomial  $g(x): A \rightarrow  A,$ $g(x)\circ f(x)$ is a  permutation polynomial if and only if for $i=1, 2, \cdots, t,$  $\psi_i(x) \circ g^{-1}(x)$ are surjections from $A$ to $S_i.$
\end{theorem}
\begin{proof}
	If $g(x)$ is a permutation polynomial, then $\psi_i(x) \circ g^{-1}(x): A \rightarrow S_i $ are surjections by assumption, so $\psi_i(x) \circ g^{-1}(x) \circ g(x) \circ f(x)= \varphi_i(x)$ are surjective from $A$ to $S_i.$
	
	Conversely, assume that  $\varphi_i(x)=\psi_i(x) \circ g^{-1}(x) \circ g(x) \circ f(x): \gf_q \rightarrow S_i  $ are surjective.  Note that $\varphi_i(x)=\psi_i(x) \circ f(x)$ and $f(x)$ is a  permutation polynomial  if and only if $\varphi_i(x)$ are surjective from $A$ to $S_i.$ So, $f(x)$ is a permutation polynomial. Hence, $g(x) \circ f(x)$ is a permutation polynomial,  which is desired result.
\end{proof}

Based on the above results, we introduce the concept of local permutation polynomials with respect to special mappings over finite fields. 

\begin{definition} \label{delo}
	For a  prime power $q,$ let  $f(x)$ be a polynomial over $\gf_q$.  Assume  that $S_i$ are nonempty finite sets of $\gf_q$ with $| S_i | \leq q/2$ and  $\psi_i(x)$ are surjections from $\gf_q$ to $S_i$ $(i=1, 2, \cdots , t).$ If $f(x)$ is a permutation polynomial over $\gf_q$ if and only if  $\varphi_i(x)=\psi_i(x) \circ f(x)$ $(i=1, 2, \cdots , t)$ are surjective from $\gf_q$ to $S_i$,  then the polynomial $f(x)$ is called the local permutation polynomial over $\gf_q$ with respect to $\psi_1(x), \psi_2(x)  , \cdots, \psi_t(x).$ 	
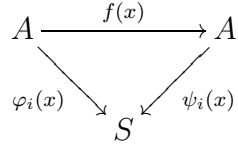
	$$\xymatrix{
		A \ar[rr]^{f(x)}\ar[dr]_{\varphi_i(x)}& & A\ar[dl]^{\psi_i(x)}\\
		&S
	}$$
	
\end{definition}

If we aim  to  prove that a polynomial is permutation, we only need to find  mappings  $\psi_i(x)$$(i=1, 2, \cdots , t)$ such that mappings  $\psi_i(x)$ ($i=1, 2, \cdots, t$) such that  the compositions $\psi_i(x)\circ f(x)$ are  surjective from $\gf_{q^n}$ to ${\rm Im}(\psi_i(x)).$

%
%
%


{\bf Acknowledgments}
\\

 P. Yuan was supported by the National Natural Science Foundation of China (Grant No. 12171163), Guangdong Basic and Applied Basic Research Foundation (Grant No. 2024A1515010589).
%


{\bf References}


\begin{thebibliography}{00}
		\bibitem{AGW} A. Akbary, D. Ghioca,  Q. Wang, On constructing permutations of finite fields, Finite Fields Appl. 2011, 17: 51--67.
		
\bibitem{ding2013cyclic} C. Ding, Cyclic codes from some monomials and trinomials, SIAM Journal on Discrete Mathematics, 2013, 27(4): 1977--1994.		
	%
%

%
\bibitem{ding2006family}  C. Ding, J. Yuan, A family of skew Hadamard difference sets,Journal of Combinatorial Theory, Series A, 2006, 113(7): 1526--1535.


\bibitem{ding2014binary} C. Ding, Z. Zhou, Binary cyclic codes from explicit polynomials over $GF (2m)$, Discrete Mathematics, 2014, 321: 76--89.
%
\bibitem{laigle2007permutation} Y. Laigle-Chapuy,Permutation polynomials and applications to coding theory, Finite Fields and Their Applications,2007, 13(1): 58--70.
%

	\bibitem{li2019compositional} K. Li, L. Qu, Q. Wang, Compositional inverses of permutation polynomials of the form $x^rh(x^s)$
over finite fields, Crptogr. Commun. 2019, 11: 279--298.
\bibitem{lidl1997finite} R. Lidl, H. Niederreiter, Finite fields, Cambridge University Press, New York, 1997.
%

\bibitem{lidl1994introduction} R. Lidl, H. Niederreiter, Introduction to finite fields and their applications, Cambridge University Press, New York, 1994.
%
		

\bibitem{NLQW21} T. Niu, k. Li, L. Qu, Q. Wang, Finding compositional inverses of permutations from the AGW criterion, IEEE Trans. Inf. Theory, 2021, 67: 4975--4985.
	
	
	
	\bibitem{park2001permutations} Y. H. Park and J. B. Lee, Permutation polynomials and group permutation polynomials, Bull. Austral. Math. Soc., 2001, 63:  67--74.
	
	
	\bibitem{rivest1978method} R.L. Rivest, A. Shamir, L. Adleman, A method for obtaining digital signatures and public-key cryptosystems, Communications of the ACM, 1978, 21(2): 120--126.


\bibitem{schwenk1998public} J. Schwenk, K. Huber, Public key encryption and digital signatures based on permutation polynomials, Electronics Letters, 1998, 34(8): 759--760.

	
	\bibitem{wan1991permutation} D. Wan, R. Lidl, Permutation polynomials of the form $x^rf(x^{(q-1)/d})$ and their group structure, Monatsh. Math., 1991, 112: 149--163.
	
	\bibitem{wang2007cyclotomic} Q. Wang, Cyclotomic mapping permutations over finite fields, in Sequences, subsequences, and consequences, ser. Lecture Notes in Comput. Sci., 2007, 4893: 119--128.
	
	\bibitem{wu2013linearized} B. Wu, Linearized polynomials over finite fields revisited, Finite Fields Appl. 2013, 22: 79--100.
	 
	\bibitem{wu2024permutation} D. Wu, P. Yuan, Permutation polynomials and their compositional inverses over finite fields by a
	local method, Designs, Codes and Cryptography,  2024, 92: 267--276. 
	
	\bibitem{wu2024some} D. Wu, P. Yuan, Some new results on permutation polynomials of the form $(x^q-x+\delta)^s+x$  over $\mathbb{F}_{q^2}$, Finite Fields and Their Applicaitons, 2024, 93: 102329. 	
 
 	\bibitem{wu2022further} D. Wu, P. Yuan, Further results on permutation polynomials from trace functions, Applicable Algebra in Engineering, Communication and Computing, 2022, 33: 341--351.
 
	
	 
	
		\bibitem{yuan2022agwpps} P. Yuan, Compositional inverses of AGW-PPs, Adv. Math. Comm. 2022, 16(4): 1185--1195.
	\bibitem{yuan2022local} P. Yuan, Local method for compositional inverses of permutation polynomials, Communications in Algebra, (Commun. Algebra,) 2024, 52(7): 3070--3080. 
	
	
	\bibitem{yuan2022permutation} P. Yuan, Permutation polynomials and their compositional inverses, 2022,  arXiv:2206.04252. to appear in Algebra Colloq. 
	
	\bibitem{yuan2011note} P. Yuan, X. Zeng, A note  on linear permutation polynomials, Finite Fields Appl. 2011, 17: 488--491.
	
	\bibitem{yuan2015permutation} P. Yuan, Y. Zheng, Permutation polynomials from piecewise functions, Finite Fields Appl., 2015, 35: 215--230. 
	
		\bibitem{zheng2020on} Y. Zheng, F. Wang, L. Wang, W. Wei, On inverses of some permutation polynomials over finite fields of characteristic three, Finite Fields Appl. 2020, 66: 101670.
	
	\bibitem{zheng2020onieee} Y. Zheng, Q. Wang, W. Wei, On the inverses of permutation polynomials of small degree over finite fields, IEEE Trans. Inf. Theory, 2020, 66(2): 914--922.	
	
	\bibitem{zieve2009} M. E. Zieve, On some permutation polynomials over $\gf_q$ of the form $x^rh(x^{(q-1)/d}),$ Proc. Am. Math. Soc., 2009, 137(7): 2209--2216.
	 

	
	
%
%
%
%
%
%

%
%
%
%
%
%
%
%
%
%
%
%
%
%
%


\end{thebibliography}
\end{document}